\documentclass[10pt]{amsart}
\usepackage{amscd, amssymb, amsfonts, amsthm, amsmath}
\usepackage[all]{xy}
\usepackage{color}
\usepackage{mathdots}
\usepackage{hyperref}

\DeclareMathOperator{\add}{add}
\DeclareMathOperator{\op}{op}
\DeclareMathOperator{\soc}{Soc}
\DeclareMathOperator{\topp}{Top}
\DeclareMathOperator{\length}{l}
\DeclareMathOperator{\start}{s}
\DeclareMathOperator{\target}{t}
\DeclareMathOperator{\fin}{findim}
\DeclareMathOperator{\prj}{P}
\newcommand{\fidim}{\mbox{\rm{$\phi$dim}}}
\newcommand{\psidim}{\mbox{\rm{$\psi$dim}}}

\def\mod{\mbox{\rm{mod}}}
\def\ind{\mbox{\rm{ind}}}
\def\add{\mbox{\rm{add}}}

\def\pd{\mbox{\rm{pd}}}
\def\id{\mbox{\rm{id}}}
\def\gd{\mbox{\rm{gldim}}}

\def\rk{\hbox{\rm{rk}}}

\begin{document}
\newcommand{\mono}[1]{%
\gdef\puA{#1}}
\newcommand{\puA}{}
\newcommand{\faculty}[1]{%
\gdef\puC{#1}}
\newcommand{\puC}{}
\newcommand{\facultad}[1]{%
\gdef\puD{#1}}
\newcommand{\puD}{}
\newcommand{\N}{\mathbb{N}}
\newcommand{\Z}{\mathbb{Z}}
\newcommand{\gl}{\mbox{\rm{gldim}}}
\newtheorem{teo}{Theorem}[section]
\newtheorem{prop}[teo]{Proposition}
\newtheorem{lema}[teo] {Lemma}
\newtheorem{ej}[teo]{Example}
\newtheorem{obs}[teo]{Remark}
\newtheorem{defi}[teo]{Definition}
\newtheorem{coro}[teo]{Corollary}
\newtheorem{nota}[teo]{Notation}

\newenvironment{note}{\noindent Notation: \rm}

%\newenvironment{proof}{\noindent {\it\bf Proof: \newline}\rm}

%\predate{\begin{flushright}\large\scshape}
%\postdate{\par\end{flushright}}

\title[Igusa-Todorov $\phi$ function for truncated path algebras]{Igusa-Todorov $\phi$ function for truncated path algebras}

\author[Barrios]{Marcos Barrios}
\address{Universidad de La Rep\'ublica,
Uruguay} \email{marcosb@fing.edu.uy}

\author[Mata]{Gustavo Mata}
\address{Universidad de La Rep\'ublica,
Uruguay} \email{gmata@fing.edu.uy}

\author[Rama]{Gustavo Rama}
\address{Universidad de la República\\
Uruguay} \email{grama@fing.edu.uy}

\maketitle

\begin{abstract}
Given a truncated path algebra $A=\frac{\Bbbk Q}{J^k}$ we prove that $\fidim A = \fidim A^{\op}$.
We also compute the $\phi$-dimension of $A$ in function of the $\phi$-dimension of $\frac{\Bbbk Q}{J^2}$ when $Q$ has no sources nor sinks.
This allows us to bound the $\phi$-dimension for truncated path algebras. Finally, we characterize $A$ when its $\phi$-dimension is equal to $1$.
%We prove that for $A$ a truncated path algebra $\fidim A = \fidim A^{\op}$. 
%And compute, when its quiver $Q$ has no sources nor sinks, the $\phi$-dimension of $A$ in function of the $\phi$-dimension of the radical square zero algebra with the same associated quiver. 
\end{abstract} 

\section{Introduction}

One of the most important conjectures in the Representation Theory of Artin Algebras is the finitistic conjecture. It states that $\sup\{\pd(M): M$ is a finitely generated module of finite projective dimension$\}$  is finite. In an attempt to prove the conjecture Igusa and Todorov defined in  \cite{IT} two functions from the objects of $\mod A$ (the category of right finitely generated modules over an Artin algebra $A$) to the natural numbers, which generalizes the notion of projective dimension. Nowadays they are known as the Igusa-Todorov functions, $\phi$ and $\psi$. One of its nicest features is that they are finite for each module, and allow us to define the $\phi$-dimension and the $\psi$-dimension of an algebra. These are new homological measures in the module category. In particular it holds that
\[\fin(A)\leq \fidim(A) \leq \psidim(A)\leq\gl(A),\]
and they all agree in the case of algebras with finite global dimension.\\
Recently, various works were dedicated to study and generalize the properties of these functions. See for instance \cite{HL}, \cite{LMM}, \cite{kn: LM}. In \cite{LMM} the authors compute $\phi$ and $\psi$ for radical square zero algebras and their respective dimensions. In this article we generalize some of those results.\\ 
This article is organized as follows: the introduction and the preliminary section are devoted to fixing the notation and recalling the basic facts needed in this work. In section $3$ we show how to construct radical square zero algebras with maximal $\phi$-dimension.
In section $4$ we work with truncated path algebras. We generalize some results for the $\phi$-dimension made in \cite{LMM} for radical square zero algebras to truncated path algebras. We compute the $\phi$-dimension explicitly for algebras such that its quiver has neither sinks nor soureces, we give bounds for the $\phi$-dimension and give a proof that the $\phi$-dimension of an algebra and its opposite are equal. Finally, in section $5$ we give a characterization of truncated path algebras with $\phi$-dimension $1$.

\begin{footnotesize}
Corresponding author: Gustavo Mata.
\end{footnotesize}

\section{preliminaries}

Let $A$ be a finite dimensional basic algebra defined over a field $\Bbbk$. The category of finite dimensional right $A$-modules will be denoted by $\mod A$, the indecomposable modules of $A$ by $\ind A$, and the set of isoclasses of simple $A$-modules by $\mathcal{S} (A)$. If $S \in \mathcal{S} (A)$, $\prj(S)$ denotes the indecomposable projective associated to $S$. For a $A$-module $M$ we denote by $\soc(M)$ its socle and by $\topp(M)$ its top.\\ 
Given an $A$-module $M$ we denote its projective dimension by $\pd(M)$, its injective dimension by $\id(M)$ and the $n^{th}$-syzygy by $\Omega^n(M)$. We recall that the global dimension of $A$, which we denote by $\gd(A)$, is the supremum  of the set of projective dimensions of $A$-modules. The global dimension can be a natural number or infinity. The finitistic dimension of $A$, denoted by $\fin(A)$, is the supremum of the set of projective dimensions of $A$-modules with finite projective dimension.\\
If $Q = (Q_0,Q_1,\start,\target)$ is a finite connected quiver, $\mathfrak{M}_Q$ denotes its adjacency matrix and $\Bbbk Q$ its path algebra. Given $\rho$ a path in $\Bbbk Q$, $\length(\rho)$, $\start(\rho)$ and $\target(\rho )$ denote the length, start and target of $\rho$ respectively. $C^n$ denotes the oriented cycle graph with $n$ vertices.   

\subsection{Igusa-Todorov $\phi$ function} We recall the definition of the Igusa-Todorov $\phi$ function and some basic properties. We also define the $\phi$-dimension of an algebra.

\begin{defi}
Let $K_0(A)$ be the abelian group generated by all symbols $[M]$, where M is a f.g.
$A$-module, modulo the relations
\begin{enumerate}
  \item $[M]-[M']-[M'']$ if  $M \cong M' \oplus M''$,
  \item $[P]$ for each projective.
\end{enumerate}
\end{defi}

\noindent Let $\bar{\Omega}: K_0 (A) \rightarrow K_0 (A)$ be the group endomorphism induced by $\Omega$, and let $K_i (A) = \bar{\Omega}(K_{i-1}(A))= \ldots = \bar{\Omega}^{i}(K_{0} (A))$. Now, if $M$ is a finitely generated $A$-module then $\langle \add M\rangle$ denotes the subgroup of $K_0 (A)$ generated by the classes of indecomposable summands of $M$.

\begin{defi}\label{monomorfismo}
The \textbf{(right) Igusa-Todorov function} $\phi$ of $M\in \mod A$  is defined as 
\[\phi_{A}(M) = \min\left\{l: \bar{\Omega}_{A}|_{{\bar{\Omega}_{A}}^{l+s}\langle add M\rangle}\text{ is a monomorphism for all }s\in \mathbb{N}\right\}.\]
%l:\bar{\Omega}_{A}{|}_{{\bar{\Omega}_{A}}^{l+s}\langle add M\rangle\right\}$ is a monomorphism for all $s \in \mathbb{N}\right\}$.
\end{defi}

In case that there is no possible misinterpretation we denote by $\phi$ the Igusa-Todorov function $\phi_A$.

\begin{prop}[\cite{IT}, \cite{HLM}] \label{it1} \label{Huard1}
Given $M,N\in$ mod$A$. 

\begin{enumerate}
  \item $\phi(M) = \pd (M)$ if $\pd (M) < \infty$.
  \item $\phi(M) = 0$ if $M \in \ind A$ and $\pd(M) = \infty$.
  \item $\phi(M) \leq \phi(M \oplus N)$.
  \item $\phi\left(M^{k}\right) = \phi(M)$ for $k \in \mathbb{N}$.
  \item If $M \in \mod A$, then $\phi(M) \leq \phi(\Omega(M))+1$.
\end{enumerate}
\begin{proof}
For the statements 1-4 see \cite{IT}, and for 5 see \cite{HLM}.
\end{proof}
\end{prop}

\begin{defi} The \textbf{$\phi$-dimension} of an algebra $A$ is defined as
$$\fidim(A) = \sup\{\phi(M): M\in modA\}.$$
\end{defi}

\begin{teo}[\cite{HL}] \label{Phi = 0}

If $A$ is an Artin algebra then $A$ is self-injective if and only if $\fidim (A) = 0$.

\end{teo}

\subsection{Radical square zero algebras}
Given a quiver Q we denote by  $J$  the ideal generated by the arrows in $\Bbbk Q$.
By a radical square zero algebra we mean an algebra which is isomorphic to an algebra of the type $A=\frac{\Bbbk Q}{J^2}$.
If $\mathcal{S}({A}) = \left\{S_1, \ldots, S_n\right\}$ denotes a complete set of simple $A$-modules up to isomorphism, then we divide the set $\mathcal{S}(A)$ in the three distinct sets

\begin{itemize}
  \item $\mathcal{S}_I$ the set of injective modules in  $\mathcal{S}$,
  \item $\mathcal{S}_P$ the set of projective modules in $\mathcal{S}$,
  \item $\mathcal{S}_D = \mathcal{S}\setminus (\mathcal{S}_I \cup \mathcal{S}_P)$.
\end{itemize}

\begin{obs} \label{fgd}
For radical square zero algebras it holds that  
$\displaystyle \Omega (S_i) = \bigoplus_{\alpha:i\rightarrow j} S_j$, i.e.: $\Omega (S_i)$ is a direct sum  of simple modules, and the number of summands isomorphic to $S_j$ coincides with the number of arrows starting in the vertex $i$ and ending at the vertex $j$. 
Given a radical square zero algebra $A$ with $n$ vertices and finite global dimension, is easy to compute its global dimension using the last fact. Explictly we have
\[ \gd(A) = \sup\left\{\length(\rho): \start(\rho)\text{ is a source, and }\target(\rho)\text{ is a sink}\right\},\]
conluding that $\gd(A)$ must be less or equal to $n-1$.
%Using the last fact it is easy to compute that for a radical square zero algebra $A$ with
%$n$ vertices and finite global dimension, its global dimension must be less or equal to $n-1$.
\end{obs}

\begin{obs}

For a radical square zero algebra $A = \frac{\Bbbk Q}{J^2}$, $\Omega(M)$ is a semisimple $A$-module for every $A$-module $M$, and $K_1 (A)$ has basis $\left\{[S]_{S \in \mathcal{S}_D}\right\}$.
In particular if $Q$ has no sinks nor sources, then $K_1 (A)$ has basis $\{[S]_{S \in \mathcal{S}}\}$.

\end{obs}

\begin{defi}

Let $ A = \frac{\Bbbk Q}{J^2}$ be a radical square zero algebra, where $Q$ is finite with $n$ vertices and without sources nor sinks. We define $T:\mathbb{Q}^n \rightarrow \mathbb{Q}^n$ as the linear transformation given by $T(e_i) =  \sum_{j = 1}^n |\{\alpha:i\rightarrow j\}|e_j$.  

\end{defi}

\begin{obs}
Given a radical square zero algebra $A$ with $Q$ as the previous definition, the matrix of $T$ in the canonical basis and the matrix of $\bar{\Omega}|_{K_1(A)}$ in the canonical basis $\{[S]_{S \in \mathcal{S}}\}$ agree.
\end{obs}

The following example, from \cite{LMM}, shows that for $A = \frac{\Bbbk Q}{J^2}$, a radical square zero algebra with $|Q_0| = n$, $\fidim(A)$ can be equal to  $m$ for all $m$ between $1$ and $n-1$.   

\begin{ej}\label{ejemplo_todas_las_dimensiones}
Let $n\in\N$ and $0<m\leq n$ consider $\Gamma^m$ the quiver

$$
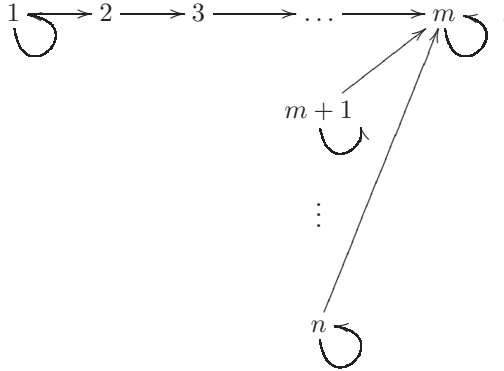
\xymatrix{  1 \ar@(d,r) \ar[r]& 2 \ar[r] & 3 \ar[r]& \ldots \ar[r]& m \ar@(d,r)\\
&&& m+1 \ar@(d,r) \ar[ur]&\\
&&& \vdots &\\
&&&n \ar@(d,r) \ar[uuur]}\ \ \ \ ,$$ 
\vspace{.5cm}

\noindent
that has $|\left(\Gamma^m\right)_0| = n$ and is without sinks nor sources. We have that
$$\fidim \left(\frac{\Bbbk \Gamma^m}{J^2}\right)= m-1.$$

\end{ej}

%\section{Maximal $\phi$-dimension on radical square zero algebras}
\section{Construction of radical square zero algebras with maximal $\phi$-dimension }

In \cite{LMM} the authors prove that a radical square zero algebra $A = \frac{\Bbbk Q}{J^2}$ has $\phi$-dimension less or equal to $n$ (Corollary 4.17) and if $\fidim (A) = n$ then $Q$ is a strongly connected quiver and the asociated Jordan form of the adjacency matrix of $Q$ is (Theorem 4.34)

 $$M_{\lambda} = \left(
    \begin{array}{ccccccc}
      0 & 0 & \ldots & 0 & 0 & 0 & 0\\
      1 & 0 & \ldots & 0 & 0 & 0 & 0\\
      0 & 1 &  & 0 & 0 & 0 & 0\\
      \vdots &  & \ddots & \ddots & \vdots & \vdots &\vdots \\
      0 & 0 & \ddots & 1 & 0 & 0 & 0\\
      0 & 0 & \ldots & 0 & 1 & 0 & 0\\
      0 & 0 & \ldots & 0 & 0 & 0 & \lambda\\
    \end{array}
  \right),$$
with $\lambda\geq2$ natural.

%The following procedure is to construct radical square zero algebras with maximal $\phi$-dimension.
Given an invertible matrix $A$, let $N_\lambda^A$ be
$$N^A_{\lambda} = A M_{\lambda}A^{-1}.$$

We write $A$ and $A^{-1}$ as

$$\begin{array}{cc} 

\begin{array}{ccccc}

A & = & \left(
    \begin{array}{cccc}
      a_{11} & a_{12} & \ldots & a_{1n}\\
      a_{21} & a_{22} & \ldots & a_{2n}\\
      \vdots & \vdots & \ddots &\vdots \\
     a_{n1} & a_{n2} & \ldots & a_{nn}\\
    \end{array}
  \right)\ 
& = & \left( \begin{array}{cccc}
 v_1, & v_2, & \ldots, & v_n 
\end{array}\right), \\

& & & & \\

A^{-1}& = & \left(
    \begin{array}{cccc}
      b_{11} & b_{12} & \ldots & b_{1n}\\
      b_{21} & b_{22} & \ldots & b_{2n}\\
      \vdots & \vdots & \ddots &\vdots \\
     b_{n1} & b_{n2} & \ldots & b_{nn}\\
    \end{array}
  \right) & = & \left( \begin{array}{c} 
  
 w_1 \\ w_2 \\ \vdots \\ w_n \end{array} \right). \\
\end{array} 
\end{array}
$$

We have $\langle v_i, w_{j} \rangle = \delta_{ij}$, as $A^{-1}A = Id_{n\times n}$.

We want to know when $N^{A}_{\lambda}$ is the adjacency matrix of a strongly connected quiver, i.e., when $N^{A}_{\lambda} \in M_{n\times n}(\N)$. It is clear that $(N^{A}_{\lambda})^n = A X_{\lambda} A^{-1}$ where

$$X_{\lambda} = \left(
    \begin{array}{ccccccc}
      0 & 0 & \ldots & 0 & 0 & 0 & 0\\
      0 & 0 & \ldots & 0 & 0 & 0 & 0\\
      0 & 0 &  & 0 & 0 & 0 & 0\\
      \vdots &  & \ddots & \ddots & \vdots & \vdots &\vdots \\
      0 & 0 & \ddots & 0 & 0 & 0 & 0\\
      0 & 0 & \ldots & 0 & 0 & 0 & 0\\
      0 & 0 & \ldots & 0 & 0 & 0 & {\lambda}^n\\
    \end{array}
  \right).$$

Then $(N^{A}_{\lambda})^n = {\lambda}^{n}A'A^{-1}$ where

$$A'= \left(
    \begin{array}{ccccc}
      0 & 0 & \ldots & 0 & a_{1n}\\
      0 & 0 & \ldots & 0 & a_{2n}\\
      \vdots & \vdots & \ddots & \vdots &\vdots \\
      0 & 0 & \ldots & 0 & a_{nn}\\
    \end{array}
  \right).$$

Then $(N^{A}_{\lambda})^n$ is the following matrix

$$(N^{A}_{\lambda})^n = {\lambda}^n\left(
    \begin{array}{cccc}
      a_{1n}b_{n1} & a_{1n}b_{n2} & \ldots & a_{1n}b_{nn}\\
      a_{2n}b_{n1} & a_{2n}b_{n2} & \ldots & a_{2n}b_{nn}\\
      \vdots & \vdots & \ddots &\vdots \\
      a_{nn}b_{n1} & a_{nn}b_{n2} & \ldots & a_{nn}b_{nn}\\
    \end{array}
  \right).$$

If $(N^{A}_{\lambda})$ is the adjacency matrix of a strongly connected quiver, given $i \in \{1, \ldots, n\}$ there exist $j \in \{1, \ldots, n\}$ such that $a_{in}b_{nj} > 0$, hence $a_{in} \neq 0$ for all $i \in \{1, \ldots, n\}$. Analogously $b_{nj} \neq 0$ for all $j \in \{1, \ldots, n\}$. We conclude that $a_{in}b_{nj} > 0$ for every $i,j \in \{1, \ldots, n\}$ and all the coefficients $a_{in}, b_{nj}$ can be choosen positive. 

Now we are able to give a procedure to construct every radical square zero algebra with maximal $\phi$-dimension.

\underline{Step 1} Choose $\tilde{v}_n$ and $\tilde{w}_n$ positive vectors of $\mathbb{Q}^{n}$ such that $\langle \tilde{v}_n,\tilde{w}_n\rangle = 1$. 

\underline{Step 2} Choose $\{\tilde{v}_1, \ldots, \tilde{v}_{n-1}\}$ a basis for $\tilde{w}_{n}^{\perp}$. Let $\tilde{A}$ be the matrix

$$\tilde{A} = \left( \begin{array}{cccc}
 \tilde{v}_1, & \tilde{v}_2, & \ldots, & \tilde{v}_n 
\end{array}\right). $$

And let $\tilde{A}^{-1}$ be
$$\tilde{A}^{-1} = \left( \begin{array}{c} 
 \tilde{w}_1 \\ \tilde{w}_2 \\ \vdots \\ \tilde{w}_n    
\end{array} \right).$$

\underline{Step 3} Choose $\tilde{\lambda} \in \mathbb{N}$ such that $\tilde{A} M_{\tilde{\lambda}} \tilde{A}^{-1} \in M_{n\times n}(\mathbb{Q} \setminus \mathbb{Q}^{-})$. The last fact is possible because $\tilde{A} M_{\tilde{\lambda}} \tilde{A}^{-1} = \tilde{A} M_{0} \tilde{A}^{-1} + \tilde{\lambda} \tilde{A} X_{1} \tilde{A}^{-1}$ and $\tilde{A} X_{1} \tilde{A}^{-1}$ is a positive matrix.

\underline{Step 4} Consider $m \in \mathbb{N}\backslash\{0\}$ such that $m\tilde{A} M_{\tilde{\lambda}} \tilde{A}^{-1} \in M_{n\times n}(\mathbb{Z}\setminus \mathbb{Z}^-)$, then there exist a quiver $Q$ such that $\mathfrak{M}_Q = \tilde{A} mM_{\tilde{\lambda}} \tilde{A}^{-1} $.
Finally $M_{k\tilde{\lambda}}$ is the Jordan form of $\mathfrak{M}_Q$ and we conclude that $\fidim \frac{\Bbbk Q}{J^{2}} = n$.

From step 4, it is easy to see the next remark.

\begin{obs}

If $Q$ is a quiver such that $\fidim (\frac{\Bbbk Q}{J^2})$ is maximal, then for all $m \in\mathbb{Z}^+$, $ m \mathfrak{M}_Q$  is the adjacency matrix of a quiver $\bar{Q}$ with $\fidim (\frac{\Bbbk \bar{Q}}{J^2})$ maximal.

\end{obs}

Step 3 can be adapted so that $\tilde{A} M_{\tilde{\lambda}} \tilde{A}^{-1} \in M_{n\times n}(\mathbb{Q^+})$. Using that fact and the previus remark we get the remark

\begin{obs}

For every quiver $\bar{Q}$ there exist a quiver $Q$ such that $\bar{Q}$ is a subquiver of $Q$ and $\fidim (\frac{\Bbbk \bar{Q}}{J^2})$ is maximal.
\end{obs}

\section{Truncated path algebras}

An Artin algebra $A$ is called monomial if $A \cong \frac{\Bbbk Q}{I}$, where $I$ is generated by paths of $Q$. Truncated path algebras are a special case of monomial algebras. An Artin algebra $A$ is a truncated path algebra if there exists a quiver $Q$ and $k \in \N$ such that $A \cong \frac{\Bbbk Q}{J^k}$.

\begin{obs}

If $A$ is a truncated path algebra and $\rho$, $\nu$ paths of $\Bbbk Q$, then

\begin{itemize}

\item $\rho A = \nu A$ if $l(\rho) = l(\nu)$ and $t(\rho) = t(\nu)$,

\item If $Q$ has no sinks nor sources then $\rho A = \nu A$ if and only if $l(\rho) = l(\nu)$ and $t(\rho) = t(\nu)$.

\end{itemize}

\end{obs}

We denote by $M^l_v(A)$ the ideal $\rho A$, where $l(\rho) = l$, $t(\rho) = v$ and $M^l(A) = \oplus_{v \in Q_0} M^l_v(A) $.

Given $0<l\leq k-1$, it is easy to see that $M^l_v(A)$ exists if and only if $\id_{\frac{\Bbbk Q}{J^{2}}} S_v \geq l$. Also, $M^l_v(A)$ is a projective $A$-module if and only if $\pd_{\frac{\Bbbk Q}{J^{2}}} S_v \leq k-l-1$. 
We also denote by $V_l(A)$ the subgroup of $K_0(A)$ with basis $\left\{\left[M^l_v(A)\right]:v \in Q_0,\ \left[M^l_v(A)\right]\neq 0\right\}$.

\begin{obs}
If $A = \frac{\Bbbk Q}{J^k}$ is a truncated path algebra then
%$\Omega( M_v^l(A)) = \bigoplus \left\{M_w^{k-l}(A):\rho : \ \start(\rho) = v,\ \target(\rho) = w,\ \length(\rho)= k-l \right\}.$$
$$\Omega( M_v^l(A)) = \bigoplus_{\rho:\left\{\substack{\start(\rho) = v\\ \length(\rho)= k-l}\right.} M_{ \target(\rho)}^{k-l}(A),$$
$$\Omega^2( M_v^l(A)) = \bigoplus_{\rho:\left\{\substack{\start(\rho) = v\\ \length(\rho)= k}\right.} M_{ \target(\rho)}^{l}(A).$$
%$$\Omega( M_v^l(A)) = \bigoplus_{\rho:\begin{array}{c}\start(\rho) = v\\
%\target(\rho) = w\\ 
%\length(\rho)= k-l 
%\end{array}
%} M_w^{k-l}(A).$$
%$$\Omega^2( M_v^l(A)) = \bigoplus_{\rho : \ \start(\rho) = v,\ \target(\rho) = w,\ \length(\rho)= k} M_w^{l}(A).$$
\end{obs}

The following result allows us to compute the global dimension of truncated path algebras as a function of the length of the longest path of their quivers. For a proof see Theorem 5 of \cite{DH-ZL}.
\begin{prop}\cite{DH-ZL}

Let $A  \cong \frac{\Bbbk Q}{J^k}$ be a truncated path algebra. Then the following statements hold.

\begin{enumerate}

\item $\gl A < \infty$ if and only if $Q$ has no oriented cycles. 

\item If $A$ has finite global dimension, then 
$$\gl A =
\left\{
\begin{array}{ll}
2\frac{l}{k}&\text{if } l\equiv0\pmod{k},\\
2\left\lfloor\frac{l}{k}\right\rfloor+1&\text{otherwise},
\end{array}\right.$$ where $l$ is the length of the largest path of $\Bbbk Q$.

\end{enumerate}

%\begin{proof}\ 
%\begin{enumerate}
%\item Suppose that $Q$ has an oriented cycle $C^{m}$. Let $i_0$ be a vertex in $C$, then $S_{i_0}$ is a direct summand of $\Omega^{2km}( S_{i_0})$. From the last fact it follows that $\pd S_{i_0} = \infty$. On the other hand, if there is no oriented cycle in $Q$ then every path has lenght less or equal to $l$, thus every simple module has finite projective dimension. 
%\item There are three cases for the first syzygy of simple modules:
%\begin{itemize}
%\item If every path from $i_0$ to a sink has length less or equal to $k-1$, then $\Omega(S_{i_0})$ is a projective module, i.e., $\pd (S_{i_0}) = 1$.
%\item If there exist some paths from $i_0$ to a sink with length equal to $k$ and the rest has length less than $m$, then there exist a non-projective direct summand of $\Omega(S_i)$ and $\Omega^2(S_{i_0})$ is projective, i.e., $\pd (S_{i_0}) = 2$. 
%\item If there exist paths from $i_0$ to a sink with length bigger than $k$, then $\Omega^2(S_{i_0})$ has non-projective summands. Let $i_0$ be a vertex which is the start of a maximal path, i.e., a path with length $l$. If $l\not\equiv0\pmod{k}$ and $l>k$, then $\Omega^{2\left\lfloor\frac{l}{k}\right\rfloor}(S_{i_0}) = \oplus_{i\in I} S_i$, where every path $\rho$ from $i$ to a sink has $1 \leq l(\rho) \leq k-1$. Now, if $l\equiv0\pmod{k}$ and $l>k$, then $\Omega^{2\frac{l}{k}}(S_{i_0}) = \oplus_{i\in I} S_i$, where $S_i$ is a projective module and $\Omega^{2\frac{l}{k}-1}(S_{i_0})$ has non-projective summands.
%\end{itemize}

%\end{enumerate}

%\end{proof}

\end{prop}

For a proof of the next theorem see Theorem 5.11 of \cite{BH-ZR}, and for definitions of skeleton and $\sigma$-critical see \cite{DH-ZL}.

\begin{teo}\cite{BH-ZR}\label{sizigia en truncadas}
Let $A$ be a truncated path algebra. If $M$ is any nonzero left $A$-module with
skeleton $\sigma$, then
$$\Omega(M) \cong \bigoplus_{\rho \text{ is } \sigma \text{-critical}} \rho A.$$
\end{teo}

The next result follows from the above theorem.
\begin{prop}\label{sizigia de truncadas}
If $A = \frac{\Bbbk Q}{J^k}$ is a truncated path algebra, then the following set $\{[\rho A]: l(\rho) \geq 1,\ [\rho A]\neq 0\}$ is a basis for $K_1(A)$.
\end{prop}
%aca quedamos

\begin{proof}
As a consequence of Theorem \ref{sizigia en truncadas}, $K_1(A) \subset \langle \{[\rho A ]\}_{ l(\rho ) \geq 1}\rangle $. On the other hand, given $\rho$ a path of length bigger or equal to $1$, if $\rho A$ is non-projective then $\rho A$ is a submodule of $s(\rho)A$. Thus $\rho A = \Omega \left(\frac{s(\rho)A}{\rho A}\right)$.  
\end{proof}

For a non self-injective truncated path algebra we can compute its $\phi$-dimension as follows.

\begin{coro}\label{calculo de la phi}

If $A = \frac{\Bbbk Q}{J^k}$ is a non self-injective truncated path algebra, then 

$$\fidim(A) = 1+\phi(\bigoplus_{ l(\rho) \geq 1} \rho A).$$

\end{coro}

Now, if $A$ is a self-injective truncated path algebra, we have the result that generalize Proposition 4.13 of \cite{LMM}.

\begin{prop}\label{autoinyectiva} For a connected non-simple truncated path algebra $A = \frac{\Bbbk Q}{J^k}$ the following statements are equivalent.
\begin{enumerate}
\item $M^l_v$ exists, is not a projective module and $\Omega(M^l_v)$ is indecomposable for all $v \in Q_0$ and $1 \leq l \leq k-1$.
\item $Q$ is a cycle, i.e., $Q = C^n$.
\item $A$ is a Nakayama algebra without injective simple modules.
\item $A$ is a Nakayama algebra without projective simple modules.
\item $A$ is a selfinjective algebra.
\item All indecomposable projective modules have length $k$.
\item All indecomposable injective modules have length $k$.
\item The $\phi$-dimension of $A$ is zero.
\end{enumerate}
\end{prop}
\begin{proof}
We start by showing that the first statement implies the second one.
Since we are assuming that $M^l_v$ exist for all $v \in Q_0$ then $v$ is not a source for all $v \in Q_0$. In an analogous way, since $M^l_v$ is not projective for all $v \in Q_0$ then $v$ is not a sink for all $v \in Q_0$. On the other hand, for every vertex $v$ there is only one arrow that starts at $v$ because $\Omega(M^l_v)$ is indecomposable for all $v \in Q_0$. By the pidgeonhole principle for every vertex $v$ there is only one arrow with target $v$. Because the quiver $Q$ is finite we get that $Q$ is a cycle $C^n$.

It is clear that the second statement implies all the other statements.

We now show that the third statement implies the second one. This follows from the fact that the quiver of a Nakayama algebra is either a linearly ordered $A_n$ or a cycle $C^{n}$. If there is no simple injective it must be a cycle.

Equivalence between two and four is similar to the equivalence between two and three.

Finally we show that the fifth statement implies the second one. The hypothesis implies that there is exactly one arrow starting and one arrow ending at any vertex. Therefore $Q$ is a cycle $C^n$.

The sixth and seventh statements easily implies the second one.

And finally, the eigth and fifth statements are equivalent following Corollary 6 from \cite{HL}.
\end{proof}

\begin{nota}
Let $A = \frac{\Bbbk Q}{J^k}$ be a truncated path algebra with quiver $Q$. If $Q_0 = \{1, \ldots, n\}$ and $1 \leq l \leq k-1$, we denote by $I_l$ the following set of vertices of $Q$
$$I_l = \left\{ i \in Q_0 : \id_{\frac{\Bbbk Q}{J^{2}}} S_i \leq l-1  \hbox{ and } \pd_{\frac{\Bbbk Q}{J^{2}}} S_i \leq k-l-1 \right\}.$$
\end{nota}

\begin{nota}
Let $\mathfrak{N}$ be a square matrix in $M_n({\Bbbk})$ and $I$, $J \subset \{1, \ldots, n\}$. We denote by $\mathfrak{N}_{I,J}$ the matrix in $M_{n-|I|\times n-|J|}(\Bbbk)$ obtained from $\mathfrak{N}$ removing the rows from $I$ and the columns from $J$. 
\end{nota}

\begin{obs}
If $A = \frac{\Bbbk Q}{J^k}$ is a truncated path algebra, then the associated matrix of $\bar{\Omega}|_{V_1 \oplus V_2 \oplus \ldots \oplus V_{k-1}}$ is
$$ \left( \begin{array}{ccccc} 
0 & \ldots & \ldots & 0 & \left(\mathfrak{M}_Q^{k-1}\right)_{I_1,I_{k-1}} \\
\vdots &  & \iddots & \left(\mathfrak{M}_Q^{k-2}\right)_{I_2,I_{k-2}} & 0 \\
\vdots & \iddots & \iddots & \iddots & \vdots \\
0 & \left(\mathfrak{M}_Q^2\right)_{I_{k-2},I_{2}} & \iddots &  & \vdots \\
\left(\mathfrak{M}_Q^1\right)_{I_{k-1},I_{1}} & 0 & \ldots & \ldots & 0 \\
\end{array} \right),$$
where $V_l = \left\langle \left\{ \left[M_v^l(A)\right]\right\}_{v\in Q_0\setminus I_l} \right\rangle$ and $\mathfrak{M}_Q$ is the adjacency matrix of the quiver $Q$.
\end{obs}

\begin{obs}\label{maximo}
For a truncated path algebra $A = \frac{\Bbbk Q}{J^k}$ we have the following equallities
$$\phi(\bigoplus_{l(\rho) \geq 1}\rho A)  = \max_{1\leq l\leq m-1} \left\{ \phi\left(M^l(A)\right)\right\} = \max_{1\leq l\leq m-1}\left\{\phi\left(M^{l}(A) \oplus M^{k-l}(A)\right)\right\}.$$
\end{obs}

\begin{obs} \label{commutative}
If $\mathfrak{N} = \mathfrak{M}^{t} \in M_n(\Bbbk)$, then $\rk \mathfrak{M}^s = \rk \mathfrak{N}^s$ for all $s \in \N$.
\end{obs}

For a radical square zero algebras $A$ it is known that the $\phi$-dimension is symmetric, i.e., $\fidim(A) = \fidim(A^{\op})$, see \cite{LMM}. The following result shows that the same is true for truncated path algebras.

\begin{teo}\label{simetria}

Let $A = \frac{\Bbbk Q}{J^k}$ and $A^{\op}= \frac{\Bbbk Q^{\op}}{J^k}$ be truncated path algebras, then
$$\fidim (A) = \fidim (A^{\op}).$$
\begin{proof}

By Remark \ref{maximo} $\fidim (A) = \max\{ \phi_A( M^l(A) \oplus M^{k-l}(A))\}$ for a truncated path algebra $A$.

The matrices $ [\bar{\Omega}_A|_{V_l \oplus V_{k-l}}]$ and $[\bar{\Omega}_{A^{\op}}|_{V_l \oplus V_{k-l}}]$ are
$$\begin{array}{ccc}
[\bar{\Omega}_A|_{V_l \oplus V_{k-l}}] = \left( \begin{array}{cc}
0 & \mathfrak{M} \\
\mathfrak{N} & 0
\end{array} \right), & &
[\bar{\Omega}_{A^{\op}}|_{V_l \oplus V_{k-l}}] = \left( \begin{array}{cc}
0 & \mathfrak{M}^t \\
\mathfrak{N}^t & 0
\end{array} \right) \end{array}.$$

If we change the order of the basis then
$$ [\bar{\Omega}_{A^{\op}}|_{V_{k-l} \oplus V_{l}}] = \left( \begin{array}{cc}
0 & \mathfrak{N}^t \\
\mathfrak{M}^t & 0
\end{array} \right).$$

By Remark \ref{commutative}: $\phi_A(M^l(A) \oplus M^{k-l}(A)) = \phi_{A^{\op}}(M^l(A^{\op}) \oplus M^{k-l}(A^{\op}))$, 
and the thesis follows from Remark \ref{maximo}.
%and by Remark \ref{maximo} it follows the thesis.
\end{proof}

\end{teo}

\begin{teo}\label{phi maximas}

Let $Q$ be a quiver with $|Q_0| = n$. If $A = \frac{\Bbbk Q}{J^k}$ is a truncated path algebra, then there exist a truncated path algebra $B = \frac{\Bbbk \bar{Q}}{J^k}$ such that $\fidim(A) \leq \fidim(B)$, where $\bar{Q}$ is a quiver with $|\bar{Q}| = n$, such that $Q$ is a subquiver of $\bar{Q}$ and it has neither sources nor sinks. 

\begin{proof}

Suppose that the quiver $Q$ has at least one source. Let $Q^1$ be the quiver formed by adding one loop in each vertex of $Q$ wich is a source. If $B^1 = \frac{\Bbbk Q^1}{J^k}$, then $M^l(A)\oplus M^{k-l}(A)$ is a direct summand of $M^l(B^1)\oplus M^{k-l}(B^1)$ as a $B^1$-module and $\bar{\Omega}_A|_{V_l (A) \oplus V_{k-l}(A)} = \bar{\Omega}_{B^1}|_{V_l(A) \oplus V_{k-l}(A)}$. Therefore $\phi_A(M^l(A)\oplus M^{k-l}(A)) = \phi_{B^1}(M^l(A)\oplus M^{k-l}(A)) \leq \phi_{B^1}(M^{l}(B^1)\oplus M^{k-l}(B^1))$, and finally $\fidim (A) \leq \fidim(B^1)$.

If $Q^1$ has no sinks then $B=B^1$. Complementary, if $Q^1$ has no sources but at least has one sink, then $\left(Q^1\right)^{\op}$ has no sinks but has at least one source and $\fidim (B^1) = \fidim \left(\left(B^1\right)^{\op}\right)$ by Theorem \ref{simetria} and the result follows by the previous case.
\end{proof}

\end{teo}

\subsection{Truncated path algebras of quivers without sources nor sinks} 
We now relate the $\phi$-dimensions of a truncated path algebra and a radical square zero algebra associated to the same quiver, when the quiver has no sinks nor sources.

In this case the projective and injective dimensions of every simple modules is infinite. Then $I_l$ is empty for every $l$. Therefore the following remarks.

\begin{obs} If $Q$ is a quiver without sources nor sinks and $A = \frac{\Bbbk Q}{J^k}$ is a truncated path algebra, then

$$\bar{\Omega}|_{V_1 \oplus V_2 \oplus \ldots \oplus V_{k-1}} = \left( \begin{array}{ccccc} 
0 & \ldots & \ldots & 0 & \mathfrak{M}_Q^{k-1} \\
\vdots &  & \iddots & \mathfrak{M}_Q^{k-2} & 0 \\
\vdots & \iddots & \iddots & \iddots & \vdots \\
0 & \mathfrak{M}_Q^2 & \iddots &  & \vdots \\
\mathfrak{M}_Q^1 & 0 & \ldots & \ldots & 0 \\
\end{array} \right). $$

\end{obs}

\begin{obs}\label{sizigia sin pozos ni fuentes}

If $Q$ is a quiver without sources nor sinks and $A = \frac{\Bbbk Q}{J^k}$ is a truncated path algebra, then

\begin{itemize}

\item $\bar{\Omega}^{2s}|_{V_1 \oplus V_2 \oplus \ldots \oplus V_{k-1}} = \left( \begin{array}{ccccc} 
\mathfrak{M}_Q^{sk}  & 0 & \ldots & \ldots & 0\\
0 & \mathfrak{M}_Q^{sk}  & \ddots &  & \vdots \\
\vdots & \ddots & \ddots & \ddots & \vdots \\
\vdots  &  & \ddots & \mathfrak{M}_Q^{sk}  & 0 \\
0 & \ldots & \ldots & 0 & \mathfrak{M}_Q^{sk}   \\
\end{array} \right),$

\item $\bar{\Omega}^{2s+1}|_{V_1 \oplus V_2 \oplus \ldots \oplus V_{k-1}} = \left( \begin{array}{ccccc} 
0 & \ldots & \ldots & 0 & \mathfrak{M}_Q^{sk+k-1} \\
\vdots &  & \iddots & \mathfrak{M}_Q^{sk+k-2} & 0 \\
\vdots & \iddots & \iddots & \iddots & \vdots \\
0 & \mathfrak{M}_Q^{sk+2} & \iddots &  & \vdots \\
\mathfrak{M}_Q^{sk+1} & 0 & \ldots & \ldots & 0 \\
\end{array} \right).$

\end{itemize}

\end{obs}

\begin{teo}\label{phi sin pozos ni fuentes}
Let $A=\frac{\Bbbk Q}{J^k}$ be a non self-injective truncated path algebra. If $Q$ has no sources nor sinks, then
$$ \fidim (A) = \left\{
\begin{array}{ll}
2\left(\frac{\fidim(\frac{\Bbbk Q}{J^2})-1}{k}\right)+1 & \text{if } \fidim(\frac{\Bbbk Q}{J^2}) - 1 \equiv0\pmod{k}, \\ 
2\left(\frac{\fidim(\frac{\Bbbk Q}{J^2})-2}{k}\right) + 2 & \text{if } \fidim(\frac{\Bbbk Q}{J^2}) - 1 \equiv 1\pmod{k}, \\
2\left\lceil\frac{\fidim(\frac{\Bbbk Q}{J^2})-2}{k}\right\rceil + 1  & \text{if }  \fidim(\frac{\Bbbk Q}{J^2}) - 1 \not\equiv 0,1\pmod{k}. \end{array} \right.$$

\begin{proof}

It follows from Remark \ref{sizigia sin pozos ni fuentes} and Corollary \ref{calculo de la phi}.\end{proof}

\end{teo}

The above theorem suggests the next definition.
\begin{defi}
Let $k$ be a natural such that $k\geq2$. We define $f_k:\N\rightarrow\N$ as
\[
f_k(m)=
\left\{
\begin{array}{ll}
0 & \text{if}\ m=0,\\
2\left(\frac{m-1}{k}\right)+1 & \text{if}\ m\equiv1\pmod{k},\\
2\left(\frac{m-2}{k}\right)+2 & \text{if}\ m\equiv2\pmod{k},\\
2\left\lceil\frac{m-2}{k}\right\rceil+1 & \text{otherwise}.
\end{array}
\right.
\]
\end{defi}

\begin{obs}
Note that $f_k$ is an increasing function.
\end{obs}

\subsection{$\phi$-dimensions for truncated path algebras}
The results from the previous section allows us to give bounds on the $\phi$-dimension of arbitrary truncated path algebras.

The next result follows from Theorem \ref{phi maximas}, Theorem \ref{phi sin pozos ni fuentes} and the fact that $\phi$-dimension of a radical square zero algebra
with $n$ vertices is less or equal to $n$ (Corollary 4.17 \cite{LMM}).

\begin{coro}\label{boundcoro}

If $A = \frac{\Bbbk Q}{J^k}$ is a truncated path algebra with $|Q_0| = n$, then
$$\fidim (A) \leq f_k(n).$$
Also there exists some algebra $A$ such that the bound is reached.

\end{coro}

\begin{prop}

Given $n \in \N$ and $k \geq 2$ then there exist a family of truncated path algebras $\left\{A_l : \ l \in \left\{0,1, \ldots, f_k(n)\right\} \right\}$ such that $\fidim(A_l) = l$. 

\begin{proof}

For $l = 0$ consider $A_0 = \frac{\Bbbk C^n}{J^k}$. And for $l = f_k(n)$ consider $A_l = \frac{\Bbbk Q}{J^k}$, where $Q$ verifies $\fidim\left(\frac{\Bbbk Q}{J^2}\right) = n$.

For the case $k = 2$ and the rest of the possible values of $l$ see Proposition 4.60 of \cite{LMM}. 

Now, let $k$ be greater or equal to $3$. Recall the quiver $\Gamma^m$ from Example \ref{ejemplo_todas_las_dimensiones}.
\begin{itemize}

\item For $l = 1$, consider $A_l = \frac{\Bbbk \Gamma^2}{J^k}$.

\item For $l$ even, consider $A_l = \frac{\Bbbk \Gamma^{l'}}{J^k}$, where $l' = \frac{k(l-2)}{2}+2$.

\item For $l$ odd, $l > 1$, consider $A_l = \frac{\Bbbk \Gamma^{l'}}{J^k}$, where $l' = \frac{k(l-3)}{2}+3$.

\end{itemize}

\end{proof}

\end{prop}

\section{Truncated path algebras with small $\phi$-dimension}

From Corollary \ref{boundcoro}, if $k$ is big in comparison with $|Q_0|$ the $\phi$-dimension of $\frac{\Bbbk Q}{J^k}$ is small. Specifically:

\begin{obs}\label{el_chico_que_me_gusta}
Let $A = \frac{\Bbbk Q}{J^k}$ be a truncated path algebra with $|Q_0| = n$.

\begin{itemize}

\item If $Q$ has neither sources nor sinks and $\fidim \frac{\Bbbk Q}{J^2} = 1$, then $\fidim A = 1$.

\item If $k \geq n$, $Q$ has neither sources nor sinks, and $\fidim \frac{\Bbbk Q}{J^2} = 2$, then
$$\fidim A = 2.$$

\item If $k \geq n-1\geq 2$, then $\fidim (A) \leq 3$.  

\end{itemize}
\end{obs}

The following example shows that there exists truncated path algebras with maximal $\phi$-dimension such that its asociated quiver has neither sources nor sinks, unlike the case of radical square zero algebras. See Corollary 4.40 of \cite{LMM}.

\begin{ej}

Let $\Gamma$ be a quiver with $\Gamma_0 = \{v_1 \ldots, v_m\}$ where $m \geq 3$ and $\fidim \frac{\Bbbk \Gamma}{J^2} = m$. Consider the quiver $Q$ with

\begin{itemize}

\item set of vertices $Q_0 = \Gamma_0 \cup \left\{w_1 \ldots, w_{n-m+1}\right\}$,

\item set of arrows $Q_1 = \Gamma_1 \cup \left\{ \alpha_m, \ldots, \alpha_{n-1}\right\}$,

\end{itemize} 
where $\start(\alpha_i) = w_{i+1}$, $\target(\alpha_i) = w_i$ and $\Gamma_0 \cap \left\{w_1 \ldots, w_{n-m+1}\right\} = \left\{w_{i_0}\right\}$ with $1< i_0 <n-m+1 $. Then $Q$ has source $w_1$ and sink $w_{n-m+1}$, and $A = \frac{\Bbbk Q}{J^n}$ has $\fidim (A) = 3$. By Remark \ref{el_chico_que_me_gusta} we see that the $\phi$-dimension of $A$ is maximal.

\end{ej}

Finally we give a characterization of truncated path algebras with $\phi$-dimension equal to $1$. 

\begin{teo}

Let $A = \frac{\Bbbk Q}{J^{k}}$ be a truncated algebra, then the following statements are equivalent.

\begin{enumerate}

\item $\fidim (A) = 1$.

\item The quiver $Q\neq C^n$ and
\begin{itemize}
\item $J^k = 0$, or
\item $\det(\mathfrak{M}_{\bar{Q}}) \neq 0$, where $\bar{Q}$ is the quiver obtained from $Q$ by deleting its sources and sinks.
\end{itemize}
\end{enumerate}

\begin{proof}
An easy computation proves that the second statement implies the first one.

Assume now that the first statement is true, then there are two possible cases.
\begin{itemize}
\item $Q$ has neither sinks nor sources.\\
Then, by Theorem \ref{phi sin pozos ni fuentes}, $\fidim (A) = 1$ if and only if $\fidim \left(\frac{\Bbbk Q}{J^2}\right) = 1$ if and only if $\det(\mathfrak{M}_Q) \neq 0$ and $Q \neq C^n$.
\item $Q$ has sinks or sources.\\
If every path in $Q$ has length less than $k$, then $J^k = 0$ hence $\fidim (A) = \gd (A) = 1$, because $A$ is a hereditary algebra. Suppose there is a path with length bigger or equal to $k$. This is equivalent to the existence of a non-projective module $M_v^l$. Consider $\bar{Q}$ the full subquiver of $Q$ generated by $\left\{v \in Q_0 \text{ such that } [M_v^l] \neq 0 \text{ for some l }\in \{1,\ldots, k-1\}\right\}$. Consider the truncated algebra $B = \frac{\Bbbk \bar{Q}}{J^k}$.  Since $\bar{\Omega}|_{K_1(A)} = \bar{\Omega}|_{K_1(B)}$ then $\bar{\Omega}|_{K_1(B)}$ is injective and it follows that $\bar{Q}$ has no sinks nor sources and $\det(\mathfrak{M}_{\bar{Q}}) \neq 0$.\\
We show now that the quiver $\bar{Q}$ is obtained from $Q$ by deleting every source and sink.
Consider $v_1, v_2, v_3 \in Q_0 $ and $\rho, \gamma \in Q_1$ such that $\rho : v_1 \rightarrow v_2$, $\gamma : v_2 \rightarrow v_3$. If $v_3 \in \bar{Q}_0$ ($ v_1 \in \bar{Q}_0 $) then $v_2 \in \bar{Q}_0$, because the module $M_{v_2}^1$ is not projective ($M_{v_2}^{k-1}$ is not projective).   
\end{itemize}
\end{proof}

\end{teo}

\end{document}